\documentclass[letterpaper, 10 pt, conference]{ieeeconf} 
\usepackage{amssymb,latexsym,amsmath}
\usepackage{epsfig}
\usepackage{caption}
\usepackage{caption}
 
 \usepackage{epsfig}

\usepackage{makeidx}

\allowdisplaybreaks

 \usepackage{graphicx,fancybox,latexsym,epsfig}
\usepackage{fancyhdr,amsmath,times,amsxtra,amssymb}
\usepackage{color}

\def\diam{\mathop{\rm diam}}
\def\Lip{\mathop{\rm Lip}}

\def\P{{\mathcal P}}

\newcommand{\cF}{\mathcal{F}}
\newcommand{\cN}{\mathcal{N}}

\newcommand{\Pp}{\mathcal{P}}
\newcommand{\Ll}{\mathcal{L}}

\usepackage{amsmath}
\usepackage{amsfonts}
\usepackage{array}
\usepackage{dsfont}
\usepackage{hyperref}
\usepackage{amssymb}

\usepackage{amsthm}
\usepackage{bbold}
\usepackage{caption}
\usepackage{standalone}
\usepackage{accents}
\let\labelindent\relax
\usepackage{enumitem}
\usepackage{tikz}
\usepackage{pgfplots}
\pgfplotsset{compat=1.6}
\usepgfplotslibrary{groupplots}
\usepackage{cancel}
\allowdisplaybreaks

\newtheorem{lemma}{Lemma}
\newtheorem{theorem}{Theorem}

\newtheorem{assumption}{Assumption}[section]
       \newtheorem{cor}{\bf{Corollary}}[section]
  \newtheorem{remark}{Remark}[section]

\usepackage{amssymb,latexsym,amsmath}
\usepackage{mathrsfs}
\usepackage{epsfig}
\usepackage{caption}

\pgfplotsset{soldot/.style={color=blue,only marks,mark=*}}
\pgfplotsset{holdot/.style={color=blue,fill=white,only marks,mark=*}}
\fussy

\IEEEoverridecommandlockouts

\renewcommand{\baselinestretch}{0.94}

\usepackage[utf8]{inputenc}

\begin{document}

\sloppy
\title{Mean-Field Control with a Common Hidden State under Decentralized Observations\thanks{
E. Bayraktar is partially supported by the National Science Foundation under grant DMS-2507940 and by the Susan M. Smith chair.}
}
\author{Erhan Bayraktar\thanks{E. Bayraktar is with the Department of Mathematics,
     University of Michigan, Ann Arbor, MI, USA, email:erhan@umich.edu}
     and Ali Devran Kara\thanks{A. D. Kara is with the Department of Mathematics,
     Florida State University, Tallahassee, FL, USA, email:akara@fsu.edu}}
\maketitle
\begin{abstract}
We study optimal control of a system with multiple decision makers who share a common hidden state and receive fully decentralized observations through identical channels. The
dynamics of the hidden state and the cost incurred by the agents depend on the agents' actions only through their empirical distribution. In the limit problem with infinitely many agents, the problem reduces to a single agent control problem where the agent affects the hidden state dynamics via the conditional law of the actions
given the past values of the hidden state process. We formulate this problem as a deterministic measure valued control problem over the space of policies and provide a dynamic programming recursion.

We first show that for the limiting problem randomization over the control actions is necessary for optimality.  However, randomization over the selection of policies (i.e., mixture policies) is not required. We then show that the optimal symmetric policies designed for the infinite population problem are near optimal for the finite population problem. 

In particular,  we establish convergence rates that decay with number of agents as $\frac{1}{\sqrt{N}}$, and grow exponentially with the memory length used in the policy. \end{abstract}

\section{Introduction}

Decision making for multi-agent control problems under decentralized information structures arises a broad class of  problems with fundamental practical importance. Their applications include networked control problems, energy grids, communication networks, distributed sensor models, traffic networks etc. \cite{sandell1978survey, couillet2012electrical, bagagiolo2014mean, olfatisaber2004consensus, ho1980team, mahajan2012information, arrow1979allocation, hespanha2007survey, davison1973optimal, tsitsiklis1988decentralized} Despite their powerful modeling capacity, their mathematical analysis remains very challenging. 

The goal of this paper is to develop a general framework for multi-agent mean-field team control problems where multiple agents aim to cooperatively  control a common hidden state process under local partial observations of the  state. The agents influence the state dynamics only through the empirical distribution of their joint actions, a structure that naturally arises in large population settings and allows tractable analysis as the number of agents grows.

We now present the model precisely. Suppose $N$ agents (decision-makers or controllers) act in a cooperative way to minimize a cost function, and the agents share a common state and an action space denoted by $\mathds{X}\subset \mathds{R}^m$ and $\mathds{U}\subset{\mathds{R}}$ for some $m<\infty$, for any time step $t$,  we have
\begin{align}\label{dynamics}
x_{t+1}=f(x_t,\mu_{\bf u_t},w_t)
\end{align}
for a measurable function $f$, where $\{w_t\}$ denotes an i.i.d. noise process where 
\begin{align*}
\mu_{\bf u_t}:=\frac{1}{N}\sum_{i=1}^N\delta_{u_t^i}(\cdot)
\end{align*}
denotes the empirical distribution of the control actions ${\bf u_t}=[u^0_t,\dots,u_t^N]$ of the agents.

We assume that the agents do not have access to state $x_t$ and observe the state via
\begin{align*}
y_t^i=g(x_t,v_t^i)
\end{align*}
where $g$ is some measurable function, and $v_t^i$ is some i.i.d. observation noise process. We can equivalently represent the observations with a kernel 
\begin{align*}
O(dy^i_t|x_t).
\end{align*}
At every time stage $t$, each agent receives a cost determined by a measurable stage-wise cost function $c:\mathds{X}\times \mathcal{P}_N(\mathds{U})\to \mathds{R}$, where $\mathcal{P}_N(\mathds{U})$ is the set of all empirical measures on $\mathds{U}$ constructed using $N$ dimensional action vectors.  That is, if the state and the empirical action distribution of the agents are given by $x_t,\mu_{\bf u_t}$, the agents receive
\begin{align*}
c(x_t,\mu_{\bf u_t}).
\end{align*}

We assume that the agents are decentralized and that they only have access to their own observations and past control actions.
 We define an admissible policy for an agent $i$, as a sequence of functions $\gamma^i:=\{\gamma^i_t\}_t$, where $\gamma^i_t$ is a $\mathds{U}$-valued (possibly randomized) function which is measurable with respect to the $\sigma$-algebra generated by 
\begin{align}\label{info}
I^i_t=(y^i_t,\dots,y^i_0,u^i_{t-1},\dots,u^i_0).
\end{align}
Accordingly, an admissible {\it team} policy, is defined as $\gamma:=\{\gamma^1,\dots,\gamma^N\}$, where $\gamma^i$ is an admissible policy for the agent $i$. 

The objective of the agents is to minimize the following cost function
\begin{align*}
J^N(P_0,\gamma)=\sum_{t=0}^{T-1}E_\gamma\left[c(X_t,\mu_{\bf U_t})\right]
\end{align*}
for some finite time horizon $T<\infty$ where $X_0\sim P_0\in\P(\mathds{X})$.

The optimal cost is defined by
\begin{align}\label{opt_cost}
J^{N,*}(P_0):=\inf_{\gamma\in\Gamma}J^N(P_0,\gamma)
\end{align}
where $\Gamma$ denotes the set of all admissible decentralized team policies.

Under the decentralized information structure, the optimality formulation is known to be challenging, and in particular, the symmetric policies are  suboptimal for the finite population problem.

\subsection{Related Literature.}
The contributions of the paper cover two main research areas: (i) decentralized multi-agent control under partial information structures, including existence and structural results for optimal policies, and (ii) mean-field approximations for multi-agent decision making problems exploiting symmetry among agents.

\noindent{\bf Structural results for decentralized control.}
For decentralized control problems with finitely many agents under partial information structures finding optimal team policies is particularly challenging due to information decentralization and partial observation.   \cite{yuksel2020universal, yuksel2017convex, saldi2020topology, gupta2015existence} \cite{YukselBasarBook,yuksel2024stochastic} approach the problem by developing dynamic programming methods on the space of strategic measures and establishing existence of optimal team policies. To establish existence, the authors construct appropriate topologies on the space of policies and strategic measures, and exploit the resulting convexity and compactness properties.

\noindent{\bf Mean field approximations.}
To obtain a computationally and analytically tractable model for multi-agent control problems, a common approach is to impose symmetry assumptions. The mean-field team framework, where agents are related only through the empirical distribution of their joint state, is one such case. For cooperative team problems, where agents jointly minimize a common cost function, the main focus has been on fully observed settings where agents can observe their local states and the mean-field of the agents' states. The standard method is to study the infinite population limit, where the problem reduces to a measure-valued control problem with the mean-field distribution as the state. This setting has been extensively studied in continuous time; see \cite{bayraktar2018randomized,djete2022mckean, lauriere2014dynamic,pham2017dynamic,carmona2021convergence,germain2022numerical,bayraktar2021mean,bayraktar2021solvability, sanjari2022optimality, sanjari2021optimal,lacker2017limit,fornasier2019mean,djete2022mckean1}.

This setting has also been studied in discrete-time settings, see \cite{motte2022quantitative,motte2022mean,bayraktar2025finite, bayraktar2024infinite,bayraktar2025learning, sanjari2022optimality, sanjari2021optimal,gu2021mean,gu2019dynamic,bauerle2021mean,angiuli2022unified} and references therein for the study of dynamic programming principle, learning methods, and justification of the exchangeability of agents for large (possibly infinite) agent team settings.

Among these, \cite{sanjari2022optimality} is the only work that considers a partial observation setting. There, each agent has a local hidden state observed through a noisy channel, and the local states are coupled through the empirical average of the agents' states and actions, evolving under idiosyncratic noise without common randomness. Notably, since all noise processes in this model are idiosyncratic, in the infinite population limit, the mean-field dynamics become deterministic. The authors establish that symmetric policies are approximately optimal as the number of agents grows. Our setting differs in several key respects. First, the agents share a single common hidden state process driven by common randomness, which each agent observes through its own noisy channel. Second, the agents influence the state dynamics through the full empirical distribution of their actions, not just through its mean.  
As a result, the stochasticity in the hidden state evolution remain in the infinite population limit as well, which affects  the nature of the limiting problem significantly and requires  a different analytical approach.

\subsection{Contributions}
\begin{itemize}
\item In Section \ref{limit_problem_2}, we formulate a new control problem that correspond to the infinite population problem, where the control actions are the conditional law of $u_t$ given the past values of the hidden state and the past control functions. 
\item In Section \ref{pure_policy_section}, we show that the optimal performance can be achieved without a mixture structure on the policies, and the agents only need to randomize their actions independently.
\item In Section \ref{dpp_section}, we reformulate the infinite population problem as a measure valued problem, and provide a dynamic programming principle on the space of randomized policies.
\item Finally, in Section \ref{conv_section}, we show that the value function of the finite population problem converges to the value function of the infinite population problem. Furthermore, we show that the symmetric policies that are optimal for the infinite population problem are near optimal when used in the finite population for large $N$. In both cases, we give convergence rates. 
\end{itemize}

In what follows, we will  formulate an alternative control problem that corresponds to the control of the presented system when the number of agents tends to infinity. 
\section{Mean-field Limit Problem}\label{limit_problem_2}
We define 
\begin{align*}
\Gamma_t=\{\text{all stochastic kernels from } \mathds{Y}^{t}\times\mathds{U}^{t-1} \text{ to } \mathds{U}\}. 
\end{align*}
Let $\gamma=(\gamma_t)_{t=0}^{T-1}$ be a sequence of randomized control functions with $\gamma_t\in\Gamma_t$ such that for each $t$
\begin{align*}
\gamma_t: \mathds{Y}^t\times \mathds{U}^{t-1}\to \P(\mathds{U}).
\end{align*}
For the given $\gamma=(\gamma_t)_t$, we define 
\begin{align*}
\mathcal{L}^\gamma(u_t|x_t,\dots,x_0) \in \P(\mathds{U})
\end{align*}
to be the conditional law of $U_t$ given the past hidden states $x_t,\dots,x_0$. In particular, for some test function $h$, and for a given hidden state path $x_t,\dots, x_0$ we have 
\begin{align*}
&\int_{\mathds{U}}  h(u_t) \mathcal{L}^\gamma(u_t \mid x_{[0,t]})(du_t) \\
&=\int h(u_t) \prod_{k=0}^t\gamma_k(du_k|h_k) O(dy_k|x_k)
\end{align*} 
where $h_k=(y_k,\dots,y_0,u_{k-1},\dots,u_0)$ with $h_0=y_0$.

For a given sequence of control functions, the dynamics of hidden state is given by
\begin{align*}
x_{t+1}=f(x_t, \mathcal{L}^\gamma(u_t \mid x_{[0,t]}),w_t )
\end{align*}
which can be equivalently represented by a stochastic kernel 
\begin{align*}
\mathcal{T}(dx_{t+1}|x_t, \mathcal{L}^\gamma(u_t \mid x_{[0,t]})). 
\end{align*}
We say that $\gamma^n\in \prod_{t=0}^{T-1}\Gamma_t$ converges to $\gamma \in \prod_{t=0}^{T-1}\Gamma_t$ if 
\begin{align*}
&P^n(dx_{[0,t]}, dy_{[0,t]},du_{[0,t]})\\
& = \gamma_t^n(du_t|h_t)O(dy_t|x_t)\mathcal{T}(dx_t|x_{t-1},\mathcal{L}^{\gamma^n}(u_t|x_{[0,t]}))\\
&\qquad\dots  \gamma^n_0(du_0|y_0)O(dy_0|x_0) P_0(dx_0).
\end{align*}
converges weakly to 
\begin{align*}
&P(dx_{[0,t]}, dy_{[0,t]},du_{[0,t]})\\
 &= \gamma_t(du_t|h_t)O(dy_t|x_t)\mathcal{T}(dx_t|x_{t-1},\mathcal{L}^{\gamma}(u_t|x_{[0,t]}))\\
&\qquad \dots  \gamma_0^n(du_0|y_0)O(dy_0|x_0) P_0(dx_0).
\end{align*}
for all $t=0,\dots,T-1$.

Note that it is not directly obvious that the set of all policies $\prod_{t=0}^{T-1}\Gamma_t$ is sufficient for optimality, and one may need to consider randomizations over this set via mixture policies for optimality.

Hence, we define the decision variable for the infinite population problem to be the set of all mixture policies, alternatively, we define the decision variable to a probability measure
\begin{align*}
\Pi \in \P\left(\prod_{t=0}^{T-1}\Gamma_t\right).
\end{align*}
Under a mixture policy, i.e. under some $\Pi$, we write  $\mathcal{L}^\Pi(u_t \mid x_{[0,t]})(du_t)$ to denote the law of $u_t$ given the hidden state history $x_{[0,t]}$. In particular, for the given $\Pi$, and for some test function $h$ we have 
\begin{align*}
&\int_{\mathds{U}}  h(u_t) \mathcal{L}^\Pi(u_t \mid x_{[0,t]})(du_t) \\
&=\int h(u_t) \prod_{k=0}^t\gamma_k(du_k|h_k) O(dy_k|x_k)\Pi(d\gamma).
\end{align*} 
The dynamics are then written as 
\begin{align*}
x_{t+1}=f(x_t, \mathcal{L}^\Pi(u_t \mid x_{[0,t]}),w_t ),
\end{align*}
the stage-wise cost function similarly is given by
\begin{align*}
c(x_t,\mathcal{L}^\Pi(u_t \mid x_{[0,t]})).
\end{align*}

The objective is to minimize the following cost function
\begin{align*}
J^\infty(P_0,\Pi)=\sum_{t=0}^{T-1}E_\Pi\left[c(X_t, \mathcal{L}^\Pi(U_t \mid X_{[0,t]}))\right]
\end{align*}
for some finite time horizon $T<\infty$ where $X_0\sim P_0\in\P(\mathds{X})$.

The optimal cost is defined by
\begin{align}\label{opt_cost_inf}
J^{\infty,*}(P_0):=\inf_{\Pi}J^\infty(P_0,\Pi).
\end{align}

\subsection{Optimality of Pure Policies}\label{pure_policy_section}
We now show that the mixture policies can be replicated using sequence of policies without randomization over the policy spaces.
Under the mixture policy $\Pi$, the joint distribution of actions given observations is:
\[
    \nu(d u_{[0,t]} \mid y_{[0,t]})
    = \int \prod_{k=0}^{t} \gamma_k(d u_k \mid y_{[0,k]}, u_{[0,k-1]}) \; \Pi(d \boldsymbol{\gamma}).
\]

\begin{lemma}[Replication of Mixture Policies]
\label{thm:main}
Let $\Pi \in \Pp\left(\prod_{t=0}^{T-1}\Gamma_t\right)$ be a mixture policy. Then there exists a sequence of policies 
\[
    \bar{\boldsymbol{\gamma}} = (\bar{\gamma}_0, \bar{\gamma}_1, \ldots, \bar{\gamma}_{T-1}) \in \prod_{t=0}^{T-1}\Gamma_t
\]
such that for every state sequence $x_{[0,t]}$:
\[
    \Ll^{\Pi}(u_t \mid x_{[0,t]})
    = \Ll^{\bar{\boldsymbol{\gamma}}}(u_t \mid x_{[0,t]}).
\]
That is, the conditional law of the action given the states under the mixture policy equals the conditional law under the policy sequence $\bar{\boldsymbol{\gamma}}$.
\end{lemma}

\begin{proof}

We only prove the result for two time steps, the rest follows identically by inductive arguments.

We can integrate out $(u_1, \ldots, u_t)$ from the joint distribution of the actions $u_0,\dots,u_t$ under the mixture policy such that:
\begin{align*}
   & \nu(d u_0 \mid y_{[0,t]})\\
    &= \int_{\Pi} \gamma_0(d u_0 \mid y_0)\underbrace{\left[\int \prod_{k=0}^{t} \gamma_k(d u_k \mid y_{[0,k]}, u_{[0,k-1]})\right]}_{= \, 1} \Pi(d \boldsymbol{\gamma}).
\end{align*}
Each $\gamma_k$ is a probability kernel, so the bracketed integral equals 1. Therefore,
\[
    \nu(d u_0 \mid y_{[0,t]})
    = \int_{\Pi} \gamma_0(d u_0 \mid y_0) \; \Pi(d \boldsymbol{\gamma}).
\]
The right-hand side depends only on $y_0$, not on $y_1, \ldots, y_t$. So, we define:
\[
    \bar{\gamma}_0(d u_0 \mid y_0)
    \;:=\; \int_{\Pi} \gamma_0(d u_0 \mid y_0) \; \Pi(d \boldsymbol{\gamma}).
\]
\medskip
We now construct $\bar{\gamma}_1$.
For each fixed $y_{[0,t]}$, the mixture joint $\nu(d u_{[0,1]} \mid y_{[0,t]})$ is a probability measure on $u_0 \times u_1$, and its marginal on $u_0$ by construction is $\bar{\gamma}_0(d u_0 \mid y_0)$. Furthermore, using the causality of the policies, it clearly is independent of $y_{[2,t]}$, that is  $\nu(d u_{[0,1]} \mid y_{[0,t]})=\nu(d u_{[0,1]} \mid y_{[0,1]})$. Since the spaces are Polish by assumption, we can apply the disintegration theorem: there exists a stochastic kernel $\nu(d u_1 \mid u_0, y_{[0,1]})$
such that
\[
    \nu(d u_{[0,1]} \mid y_{[0,1]})
    = \nu(d u_1 \mid u_0, y_{[0,1]}) \; \bar{\gamma}_0(d u_0 \mid y_0).
\]
Thus, we can define $ \bar{\gamma}_1(d u_1 \mid y_0, y_1, u_0)$ to be this kernel.

By repeating the same argument, we can find $(\bar{\gamma}_k)_{k=0}^t$ such that
\begin{align}\label{nu_factor}
    \nu(d u_{[0,t]} \mid y_{[0,t]})
    = \prod_{k=0}^{t} \bar{\gamma}_k(d u_k \mid y_{[0,k]}, u_{[0,k-1]}).
\end{align}
Under the mixture policy $\Pi$, the conditional law of the actions given the states can be written by integrating the joint distribution on actions against the observation channels:
\[
    \Ll^\Pi(u_{[0,t]} \mid x_{[0,t]})
    = \int \nu(d u_{[0,t]} \mid y_{[0,t]}) \prod_{k=0}^{t} O(d y_k \mid x_k).
\]
By \eqref{nu_factor}, $\nu$ can be written using $\bar{\boldsymbol{\gamma}}$:
\begin{align*}
   & \Ll^{\boldsymbol{\gamma}}(u_{[0,t]} \mid x_{[0,t]})\\
    &= \int \prod_{k=0}^{t} \bar{\gamma}_k(d u_k \mid y_{[0,k]}, u_{[0,k-1]}) \prod_{k=0}^{t} O(d y_k \mid x_k).
\end{align*}
This is  the conditional law under the  profile $\bar{\boldsymbol{\gamma}}$, and thus conclude the proof.
\qed
\end{proof}

\subsection{Measure Valued Formulation and Dynamic Programming}\label{dpp_section}
The problem can be reformulated using the joint distributions of the system's random variables. At time $t$, we define the measure-valued state
\begin{align*}
P_t := \mathds{P}(dx_{[0,t]}, dy_{[0,t]}, du_{[0,t-1]}) \in\mathcal{P}(\mathds{X}^{t+1} \times \mathds{Y}^{t+1} \times \mathds{U}^{t})
\end{align*}
as the joint distribution of the hidden states, observations, and past actions up to time $t$. The control action is represented by a measure
\begin{align*}
\Theta_t := \mathds{P}(dy_{[0,t]}, du_{[0,t]}) \in \mathcal{P}(\mathds{Y}^{t+1} \times \mathds{U}^{t+1}),
\end{align*}
corresponding to the joint distribution of observations and actions, subject to the consistency constraint
\begin{align*}
\Theta_t(dy_{[0,t]}, du_{[0,t-1]} \in A) = P_t(dy_{[0,t]}, du_{[0,t-1]} \in A)
\end{align*}
for all $A \in \mathcal{B}(\mathds{Y}^{t+1} \times \mathds{U}^{t})$. That is, $\Theta_t$ and $P_t$ agree on the marginal distribution of the information variables $(y_{[0,t]}, u_{[0,t-1]})$ available to the agent at time $t$.

By disintegrating $\Theta_t$, we recover the agent's control kernel $\gamma_t \in \Gamma_t$:
\begin{align}\label{disint}
\Theta_t(dy_{[0,t]}, du_{[0,t]}) = \gamma_t(du_t | h_t) \, P_t(dy_{[0,t]}, du_{[0,t-1]}),
\end{align}
where, as before, $h_t = (y_t, \dots, y_0, u_{t-1}, \dots, u_0)$.

We note that the conditional law $\mathcal{L}^\gamma(u_t \mid x_{[0,t]})$, introduced in Section~\ref{limit_problem_2}, can be expressed directly in terms of $P_t$ and $\gamma_t$:
\begin{align}\label{lambda}
\mathcal{L}^\gamma(u_t \mid x_{[0,t]})(\cdot) &= \int \gamma_t(\cdot \mid h_t)\, P_t(dy_{[0,t]}, du_{[0,t-1]} \mid x_{[0,t]})\nonumber\\
&=:\Lambda(P_t, \gamma_t)(\cdot | x_{[0,t]}) 
\end{align}
That is, the conditional distribution of $u_t$ given the state path is obtained by averaging the agent's policy $\gamma_t$ over the observations and past actions using the conditional distribution encoded in $P_t$.

The deterministic dynamics for the measure valued state are then given by
\begin{align}\label{inf_kernel}
P_{t+1} &= F(P_t, \gamma_t) \nonumber\\ 
&:= O(dy_{t+1} | x_{t+1}) \, \mathcal{T}(dx_{t+1} | x_t, \Lambda(P_t, \gamma_t)(\cdot | x_{[0,t]}) ) \nonumber\\
&\quad \times \gamma_t(du_t | h_t) \, P_t(dx_{[0,t]}, dy_{[0,t]}, du_{[0,t-1]}).
\end{align}
The controls up to time $t-1$, namely $\gamma_0, \dots, \gamma_{t-1}$, are implicitly embedded in $P_t$ and can be recovered through disintegration as in \eqref{disint}, while $\gamma_t$ enters explicitly through the control measure $\Theta_t$.


The measure-valued formulation above admits a dynamic programming recursion. At $t = T-1$, for any given $\gamma_0, \dots, \gamma_{T-2} \in \prod_{t=0}^{T-2} \Gamma_t$ (which are encoded in $P_{T-1}$), we define
\begin{align*}
&V^*_{T-1}(P_{T-1}) \\
&:= \inf_{\gamma_{T-1} \in \Gamma_{T-1}}\mathds{E}_{P_{T-1}} \bigg[ c\big(x_{T-1}, \Lambda(P_{T-1},\gamma_{T-1})(\cdot | x_{[0,T-1]})\big) \bigg],
\end{align*}
where the expectation is over $x_{[0,T-1]}$ with distribution determined by the marginal of $P_{T-1}$.

For a general time step $t = T-2, \dots, 0$, the value function is defined by
\begin{align}\label{dp_recursion}
&V^*_t(P_t)\nonumber\\
& := \inf_{\gamma_t \in \Gamma_t} \bigg\{\mathds{E}_{P_t} \bigg[ c\big(x_t, \Lambda(P_t, \gamma_t)(\cdot |x_{[0,t]})\big) \bigg] \nonumber\\
&\qquad\qquad\qquad\qquad\qquad + V^*_{t+1}\big(F(P_t, \gamma_t)\big) \bigg\},
\end{align}
where $F(P_t, \gamma_t)$ is the dynamics defined in \eqref{inf_kernel}.
The optimal value of the problem is then
\begin{align*}
V^*_0(\bar{P}_0),
\end{align*}
where $\bar{P}_0(dx_0, dy_0) := O(dy_0 | x_0) \, P_0(dx_0)$.

\begin{remark}
The dynamic program presented above is not directly tractable, as it involves optimization over infinite-dimensional policy spaces. Nonetheless, the formulation provides a significant simplification compared to the finite-agent regime. In the finite population setting, symmetric policies are not necessarily optimal, and thus one must define the dynamic programming recursion over the joint policy spaces of all agents. Whereas the infinite population formulation reduces the problem to a search over a single policy space at each time stage of the recursion. Further simplifications may be possible by  considering parametrized family of policies to reduce the dimension of the search space; we leave this direction for future work.
\end{remark}

\section{Convergence of Value Functions}\label{conv_section}

\begin{assumption}\label{main_assmp}
\begin{itemize}
\item[i.]  $\mathds{U}\subset \mathds{R}$ is compact.
\item[ii.] $\mathcal{T}(dx_1|x,\mu_{\bf u})$ is Lipschitz in $x,\mu_{\bf u}$ under $W_1$ such that
\begin{align*}
&W_1\left(\mathcal{T}(dx_1|x,\mu_{\bf u})-\mathcal{T}(\cdot|x',\mu_{\bf u'}\right)|\\
&\leq K_f \left(\|x-x'\|+W_1(\mu_{\bf u},\mu_{\bf u'})\right)
\end{align*} 
for some $K_f<\infty$, where $W_1$ is the first order Wasserstein distance.
\item[iii.] $O(dy|x)$ is Lipschitz in $x$ under total variation such that 
\begin{align*}
\|O(dy|x)-O(dy|x')\|_{TV}\leq K_O \|x-x'\|
\end{align*}
for some $K_O<\infty$, for all $x,x'$.
\item[iii] $c$ is  Lipschitz in $x,\mu_{\bf u}$ such that
\begin{align*}
|c(x,\mu_{{\bf u}})-c(x',\mu_{{\bf u'}})|\leq K_c \left(|x-x'|+W_1(\mu_{\bf u},\mu_{\bf u'})\right)
\end{align*}
for some $K_c<\infty$.
\end{itemize}
\end{assumption}

\begin{theorem}\label{main_thm}
Under Assumption \ref{main_assmp}, the following folds:
For any $P_0\in\P(\mathds{X})$, we have that 
\begin{align*}
\left|J^{N,*}(P_0) - J^{\infty,*}(P_0)\right| \leq \sum_{t=0}^{T-1}\frac{M_t}{\sqrt{N}}
\end{align*}
where 
\begin{align*}
&M_t:=\\
& K_c  \frac{1-K_f}{(1-K_f)^2 + K_O K_f}\left[\left(\frac{1-K_f + K_O K_f}{1-K_f}\right)^{t} - K_f^{t}\right].
\end{align*}
\end{theorem}

\begin{proof}
We first assume that $J^{N,*}(P_0) <J^{\infty,*}(P_0)$. We then define the following mixture policy 
\begin{align*}
\mathcal{L}^{\Pi}(u_t \mid x_{[0,t]})(du_t) =\int \prod_{k=0}^t\gamma_k(du_k|h_k) O(dy_k|x_k)\Pi(d\gamma)
\end{align*} 
where $\Pi(\cdot)$ is the mixture of the finite population $\epsilon-$optimal policies such that 
\begin{align*}
\Pi(\cdot) =\frac{1}{N} \sum_{i=1}^N \delta_{\gamma^i}(\cdot)
\end{align*}
for all $t$ where $\gamma_t^i$ denote the $\epsilon-$optimal control function of agent $i$ for the finite agent problem. For the infinite population problem, we use this mixture policy to get a further upper bound on the value difference. That is, we use the mixture policy $\Pi$ for the infinite population problem based on the  $\epsilon$-optimal policies $\{\gamma_t^i\}_i$, and we use the  $\epsilon$-optimal policies $\{\gamma_t^i\}_i$ for the finite population problem. Denoting these finite population team policy by ${\bf \gamma}$, under the initial assumption that $J^{N,*}(P_0) <J^{\infty,*}(P_0)$ we then have 
\begin{align*}
J^{\infty,*}(P_0)- J^{N,*}(P_0) < J^{\infty}(P_0,\Pi)- J^{N}(P_0,{\bf \gamma})+\epsilon
\end{align*}

 We let $x_t^N$ denote the hidden state variable under the finite population dynamics driven by the finite population $\epsilon$-optimal policies $\{\gamma_t^i\}_i$ such that
\begin{align*}
x^N_{t+1}=f(x_t^N,\mu^N_{u_t},w_t)
\end{align*}
where $\mu_{u_t}^N=\frac{1}{N}\sum_{i=1}^N \delta_{u_t^i}$ such that  $u_t^i\sim \gamma_t^i(\cdot\mid h_t^i)$. 

Furthermore, we let $x_t^\infty$ denote the hidden state variable driven under the mixture policy $\Pi$ under the infinite population dynamics such that 
 \begin{align*}
x^\infty_{t+1}=f(x_t^\infty, \mathcal{L}^{\Pi}(u_t|x^\infty_{[0,t]}),w'_t).
\end{align*}

We couple the noise sequences $w_t$ and $w_t'$ in a way that reaches the $W_1$ distance between the corresponding transition kernels such that
\begin{align*}
&\mathds{E}\left[\|x_{k}^N - x_{k}^\infty\right] \\
&= \mathds{E}\bigg[W_1\big(\mathcal{T}(\cdot|x_{k-1}^N,\mu^N_{u_{k-1}}),\\
&\qquad\qquad\qquad \mathcal{T}(\cdot|x_{k-1}^\infty,\mathcal{L}^{\Pi}(u_{k-1}|x^\infty_{[0,k-1]}))\big) \bigg]
\end{align*}
for all $k$. In particular, under this particular coupling, using the continuity of the transition kernel for the hidden state we have that 
\begin{align}\label{x_iter}
&\mathds{E}\left[\|x_k^N - x_k^\infty\|\right] \nonumber\\
&\leq K_f \mathds{E}\left[\|x_{k-1}^N - x_{k-1}^\infty\|\right]\nonumber\\
&\qquad  + K_f   \mathds{E}\left[W_1\left(\mu^N_{u_{k-1}},  \mathcal{L}^{\Pi}(u_{k-1}|x^\infty_{[0,k-1]})\right)\right].
\end{align}

Under this coupling of the noise sequences, we then write 
\begin{align*}
&\left|\mathds{E}\left[c(x_t^N,\mu^N_{u_t})\right] - \mathds{E}\left[c(x_t^\infty,  \mathcal{L}^{\Pi}(u_t|x^\infty_{[0,t]}))\right]\right|\\
&\leq \mathds{E}\left[\left|c(x_t^N,\mu^N_{u_t})-c(x_t^\infty,  \mathcal{L}^{\Pi}(u_t|x^\infty_{[0,t]}))\right|\right]\\
&\leq K_c\left(\mathds{E}[\|x_t^N - x_t^\infty\|] + \mathds{E}\left[W_1\left(\mu_{u_t}^N,  \mathcal{L}^{\Pi}(u_t|x^\infty_{[0,t]})\right)\right]\right)\\
&\leq  K_c\bigg(\mathds{E}[\|x_t^N - x_t^\infty\|] + \mathds{E}\left[W_1\left(\mu_{u_t}^N,  \mathcal{L}^{\Pi}(u_t|x^N_{[0,t]})\right)\right] \\
&\qquad \qquad +  \mathds{E}\left[W_1\left(\mathcal{L}^{\Pi}(u_t|x^N_{[0,t]}),  \mathcal{L}^{\Pi}(u_t|x^\infty_{[0,t]})\right)\right] \bigg).
\end{align*}
We first note that given $x^N_{[0,t]}$, the local histories, $h_t^i$, of the agents are independent. Hence, we can use Lemma \ref{dudley_lemma} to write
\begin{align*}
 \mathds{E}\left[W_1\left(\mu_{u_t}^N,  \mathcal{L}^{\Pi}(u_t|x^N_{[0,t]})\right)\right] \leq \frac{K}{\sqrt{N}}.
\end{align*}
Furthermore, using the total variation continuity of the observation channels, we can also write 
\begin{align*}
  &\mathds{E}\left[W_1\left(\mathcal{L}^{\Pi}(u_t|x^N_{[0,t]}),  \mathcal{L}^{\Pi}(u_t|x^\infty_{[0,t]})\right)\right]\\
  & \leq K_O \sum_{k=0}^t\mathds{E}\left[\|x_k^N - x_k^\infty\|\right].
\end{align*}

Denoting by $m_t:=\mathds{E}\left[W_1\left(\mu_{u_t}^N,  \mathcal{L}^{\Pi}(u_t|x^\infty_{[0,t]})\right)\right]$, one can derive 
\begin{align*}
m_t \leq \frac{K}{\sqrt{N}} + \frac{K_O K_f}{1-K_f} \sum_{j=0}^{t-1}(1-K_f^{t-j})m_j.
\end{align*}
Noting that $m_0=0$, one can solve for an upper-bound on $m_t$ such that
\begin{align*}
m_t \leq \frac{K}{\sqrt{N}}\left(\frac{1-K_f + K_OK_f}{1-K_f}\right)^{t-1}.
\end{align*}
Furthermore, using \eqref{x_iter}, we can derive
\begin{align*}
&\mathds{E}\left[\|x_k^N - x_k^\infty\|\right] \\
&\leq \frac{K}{\sqrt{N}} \cdot \frac{K_f(1-K_f)}{(1-K_f)^2 + K_O K_f} \\
&\qquad\qquad\qquad\times\left[\left(\frac{1-K_f + K_O K_f}{1-K_f}\right)^{t-1} - K_f^{t-1}\right].
\end{align*}
Combining the bounds, we can write 
\begin{align*}
&\left|\mathds{E}\left[c(x_t^N,\mu^N_{u_t})\right] - \mathds{E}\left[c(x_t^\infty,  \mathcal{L}^{\Pi}(u_t|x^\infty_{[0,t]}))\right]\right|  \\
&\leq K_c\frac{K}{\sqrt{N}} \cdot \frac{1-K_f}{(1-K_f)^2 + K_O K_f} \\
&\qquad\qquad\times \left[\left(\frac{1-K_f + K_O K_f}{1-K_f}\right)^{t} - K_f^{t}\right].
\end{align*}

For $J^{N,*}(P_0) <J^{\infty,*}(P_0)$, we let agent $i$ use $\gamma_t(du_t|h_t^i)$ at time $t$, where $(\gamma_t)_t$ is an $\epsilon-$optimal policy profile for the infinite population problem. We can then follow the identical steps as in the first case. The only difference is that we use Lemma \ref{dudley_lemma} for a fixed policy $\gamma$ that does not change with the agent identity.  Since, $\epsilon$ is arbitrary, letting $\epsilon\to 0$ concludes the proof.
\qed
\end{proof}

\begin{lemma}\label{dudley_lemma}
Given a fixed sequence of function $\gamma^1,\dots,\gamma^N$ where $\gamma^i:\mathds{S}\to \mathcal{P}(\mathds{U})$ for some space $\mathds{S}$ and where $\mathds{U}\subset \mathds{R}$ is compact. Let $D:=diam(\mathds{U})$, we have that 
\begin{align*}
&\mathds{E}\bigg[\sup_{\|h\|_{Lip}\leq 1} \bigg| \frac{1}{N}\sum_{i=1}^N \int h(u)\gamma^i(du|X^i) \\
&\qquad\qquad- \frac{1}{N}\sum_{i=1}^N \int h(u)\gamma^i(du|x)\mu(dx)\bigg|  \bigg] \leq C \frac{D}{\sqrt{N}}
\end{align*} 
for some constant $c<\infty$, for i.i.d. $X^i\sim \mu$. 
\end{lemma}

\begin{proof}

Since $ \frac{1}{N}\sum_i\gamma^i(\cdot|X^i)$ and $\frac{1}{N}\sum_i\int \gamma^i(\cdot|x)\mu(dx)$ are both probability measures, the constant part of $h$ cancels out. Hence, we can restrict the functions to the following set
\[
\mathcal{F} := \{h\colon \mathds{U}\to\mathds{R}:\|h\|_{\Lip}\le 1,\;h(u_0)=0\}
\]
for a fixed $u_0\in\mathds{U}$, so that $\|h\|_\infty \le D$ for all $h\in\cF$ where $D:= diam(\mathds{U})$.
 
Define the centered process
\begin{align*}
&Z_h := \frac{1}{N}\sum_{i=1}^N \bigl(g_i^h(X^i) - \mathds{E}[g_i^h(X^i)]\bigr),\\
& \text{where }g_i^h(x):=\int_{\mathds{U}} h(u)\,\gamma^i(du\,|\,x).
\end{align*}
We must bound $\mathds{E}[\sup_{h\in\cF}|Z_h|]$. We use  Dudley's theorem for sub-Gaussian processes (e.g. \cite[Theorem~8.1.3]{Vershynin_2026}):

For any $h,h'\in\cF$ and $x\in\mathds{S}$, we have that
$|g_i^h(x)-g_i^{h'}(x)|\le \|h-h'\|_\infty$.  By Hoeffding's lemma and independence, we have that 
\[
\mathds{E}\bigl[e^{\lambda(Z_h - Z_{h'})}\bigr] \le \exp\!\left(\frac{\lambda^2\|h-h'\|_\infty^2}{2N}\right).
\]
Hence $(Z_h)_{h\in\cF}$ is a sub-Gaussian process with respect to the metric
\begin{equation}\label{eq:sigma}
\sigma(h,h') := \frac{\|h-h'\|_\infty}{\sqrt{N}}.
\end{equation}
By Dudley's theorem for sub-Gaussian processes (e.g.\ \cite[Theorem~8.1.3]{Vershynin_2026}):
\begin{equation}\label{eq:dudley}
\mathds{E}\!\left[\sup_{h\in\cF}|Z_h|\right]
\le C\int_0^{\diam(\cF,\,\sigma)}\sqrt{\log\cN(\cF,\,\epsilon,\, \sigma)}\;d\epsilon,
\end{equation}
Note that using the Lipschitz bound of the functions we have that  $\diam(\cF,\sigma) \le 2D/\sqrt{N}$ where $D=diam(\mathds{U})$.  Furthermore, we also have that  $\cN(\cF,\epsilon,\sigma) = \cN(\cF,\epsilon\sqrt{N},\|\cdot\|_\infty)$.

By \cite[Theorem 2.7.1]{vanderVaart1996}, for the function class $\cF$ we have that,
\begin{equation*}
\log\cN(\cF,\,\epsilon,\,\|\cdot\|_\infty) \;\le\; K\left(\frac{D}{\epsilon}\right)
\qquad\text{for all } 0<\epsilon\le D,
\end{equation*}
where $K > 0$ is a constant. Using this on (\ref{eq:dudley}), we get
\begin{align*}
&\mathds{E}\!\left[\sup_{h\in\cF}|Z_h|\right]
\le C\int_0^{2D/\sqrt{N}}\sqrt{K\!\left(\frac{D}{\epsilon\sqrt{N}}\right)}\;d\epsilon\\
&= \frac{C\sqrt{K D}}{N^{1/4}}\int_0^{2D/\sqrt{N}}\epsilon^{-1/2}\,d\epsilon = \frac{2\sqrt{K}D}{\sqrt{N}}.
\end{align*}
\qed
\end{proof}

\begin{remark}
The result is stated for one dimensional action spaces. However, using the truncated version of the Dudley's theorem for sub-Gaussian processes, one can derive bounds for $\mathds{U}\subset \mathds{R}^d$ which behaves as $C_2 D\frac{\log\sqrt{1+N}}{N}$ for $d=2$, and as $C_3 D N^{-1/(d+3)}$ for $d=3$ using the same proof arguments. 
\end{remark}

 \begin{cor}[Near Optimality of Symmetric Policies]
Fix $\epsilon>0$. Let
\[
    {\boldsymbol{\gamma}} = ({\gamma}_0, {\gamma}_1, \ldots, {\gamma}_{T-1}) \in \prod_{t=0}^{T-1}\Gamma_t
\]
denote an $\epsilon$-optimal policy sequence for the infinite population problem. We let $J^{N}(P_0,{\boldsymbol{\gamma}})$ denote the accumulated cost if every agent follows the symmetric policy ${\boldsymbol{\gamma}} $ using their local observation histories.
Under Assumption \ref{main_assmp}, the following folds:
For any $P_0\in\P(\mathds{X})$, we have that 
\begin{align*}
J^{N}(P_0, {\boldsymbol{\gamma}} ) - J^{N,*}(P_0) \leq 2\sum_{t=0}^{T-1}\frac{M_t}{\sqrt{N}} + \epsilon
\end{align*}
where 
\begin{align*}
&M_t:= K_c \cdot \frac{1-K_f}{(1-K_f)^2 + K_O K_f} \\
&\qquad\qquad\times \left[\left(\frac{1-K_f + K_O K_f}{1-K_f}\right)^{t} - K_f^{t}\right].
\end{align*}
\end{cor}
\begin{proof}
We write 
\begin{align*}
&|J^{N}(P_0, {\boldsymbol{\gamma}} ) - J^{N,*}(P_0)|\\
&\leq  |J^{N}(P_0, {\boldsymbol{\gamma}} ) - J^{\infty,*}(P_0)|+ | J^{\infty,*}(P_0)- J^{N,*}(P_0)|
\end{align*}
The second term is bounded by Theorem \ref{main_thm} and the first term is bounded by the second part of the proof of the same theorem.
\end{proof}
 
 \section{Conclusion}
We studied a multi-agent mean-field control problem in which the agents share a common hidden state and control its dynamics through the empirical distribution of their actions. The problem is partially observed and decentralized; each agent observes the hidden state through its own symmetric observation channel. We showed that the infinite-population limit corresponds to a measure-valued deterministic control problem, and provided a dynamic programming recursion for this formulation. As the main contribution, we established that symmetric policies constructed for the infinite-population problem are near optimal for the finite-population problem, and derived a convergence rate for this approximation. As a future direction, we aim to develop computationally efficient methods for solving this problem by considering parametrized families of policies combined with policy gradient methods, and by developing finite-memory approximations, where policies depend only on a truncated history rather than the entire past, to reduce the dimensionality of the policy spaces. 
 
 
 
 
 

\bibliographystyle{plain}

\bibliography{AliBibliography,mfc_bibliography}

\end{document}